\documentclass[12pt]{amsart}

\usepackage{latexsym,amssymb,amsmath,amsthm,amscd,graphicx, xcolor, enumerate}

\usepackage{times}
\usepackage{microtype}
\usepackage{setspace}
\usepackage[margin=1.2in]{geometry}

\usepackage{cite}

\usepackage[colorlinks=true, pdfstartview=FitV, linkcolor=blue,
citecolor=blue, urlcolor=blue]{hyperref}

\numberwithin{equation}{section}
\allowdisplaybreaks

\newtheorem{theorem}{Theorem}[section]
\newtheorem{lemma}[theorem]{Lemma}
\newtheorem{corollary}[theorem]{Corollary}
\newtheorem{proposition}[theorem]{Proposition}

\theoremstyle{definition}

\newtheorem{remark}[theorem]{Remark}
\newtheorem{definition}[theorem]{Definition}

\theoremstyle{remark}

\numberwithin{equation}{section}

\makeatletter
\@namedef{subjclassname@2020}{\textup{2020} Mathematics Subject Classification}
\makeatother

\begin{document}

\title[Geometric Properties  of $\ell^{q(\cdot)}(L^{p(\cdot)})$]{Geometric Properties of Variable Mixed Lebesgue-Sequence Spaces}

\author{Reza Roohi Seraji$^{\ast,1}$}
\address{
$ ^1 $Department of Mathematics, Institute for Advanced Studies in Basic Sciences (IASBS), Zanjan 45137-66731, Iran}
\curraddr{}
\email{rroohi@iasbs.ac.ir}
\thanks{$^\ast$ Corresponding author.}

\author{Arash Ghorbanalizadeh$ ^2 $}
\address{
$ ^2 $Department of Mathematics, Institute for Advanced Studies in Basic Sciences (IASBS), Zanjan 45137-66731, Iran}
\curraddr{}
\email{ghorbanalizadeh@iasbs.ac.ir}



\date{}

\dedicatory{}

\commby{}


\begin{abstract}
We investigate the geometric properties of the variable mixed Lebesgue sequence space \(\ell^{q(\cdot)}(L^{p(\cdot)})\). Assuming that \(q_->1\), \(p_+<\infty\), and \(q(\cdot)\le p(\cdot)\) almost everywhere, we prove that both \(\ell^{q(\cdot)}(L^{p(\cdot)})\) and its dual space are strictly convex and smooth. Moreover, we establish that \(\ell^{q(\cdot)}(L^{p(\cdot)})\) is uniformly convex under these assumptions. We further show that the condition \(q_->1\) is sharp: when \(q_-=1\), the space fails to be uniformly convex, and its modular is not uniformly convex either. Finally, we apply these geometric results to the renormability of the associated variable-exponent Besov space \(B_{p(\cdot),q(\cdot)}^{\alpha(\cdot)}\).
\end{abstract}

 \maketitle

\section{introduction}

The geometric properties of the variable mixed Lebesgue--sequence space \( \ell^{q(\cdot)}(L^{p(\cdot)}) \) are the main object of study in this work. This space plays a fundamental role in the definition of the variable exponent Besov space \( B^{\alpha(\cdot)}_{p(\cdot),q(\cdot)} \), introduced by Almeida and H\"ast\"o \cite{AlHa2010}.

Variable Lebesgue spaces form a class of function spaces that generalize the classical Lebesgue spaces by allowing the exponent \( p \) to vary with the point \( x \). These spaces, denoted by \( L^{p(\cdot)} \), were originally introduced by Orlicz \cite{WO}. Although they share many properties with classical Lebesgue spaces, they also exhibit distinctive and subtle features. Beyond their intrinsic theoretical interest, variable Lebesgue spaces have important applications in nonlinear elasticity \cite{Zhikov2}, electrorheological fluids \cite{Ruzicka}, and image restoration \cite{LiLiPi}. Detailed accounts of variable Lebesgue and Sobolev spaces can be found in the monographs \cite{Cruz,Die} and the survey articles \cite{di,sam}, together with the references therein.

Several other function spaces with variable exponents have also been studied extensively. For example, Hardy spaces and Campanato spaces with variable exponents are discussed in \cite{Nakai}, while Besov spaces with variable exponents are investigated in \cite{AlHa2010}. The variable exponent Besov space \( B^{\alpha(\cdot)}_{p(\cdot),q(\cdot)} \) is defined via the mixed Lebesgue--sequence space \( \ell^{q(\cdot)}(L^{p(\cdot)}) \), whose properties have been explored in \cite{AlHa2010,GG,GRS,KeVy}. This space plays an important role in the study of evolution equations, including the Navier--Stokes equations \cite{H.Triebel-Heat-Nvier-Stocks}, as well as in the analysis of partial differential equations and problems in the calculus of variations with nonstandard growth conditions \cite{E.Acerbi-nonstgrowth,X.-L-PDEandBesov,S.Liang-PDEandBesov}.

Geometric properties such as strict and uniform convexity in the variable exponent setting are crucial for ensuring existence, uniqueness, and fine regularity of solutions to these classes of problems \cite{H.Triebel-Heat-Nvier-Stocks}.

The remainder of this paper is organized as follows. In Section \ref{section2}, we collect the necessary definitions and preliminary results. In Section \ref{section3}, we investigate the strict convexity of the space $\ell^{q(\cdot)}(L^{p(\cdot)})$. In Section \ref{section4}, we prove that $\ell^{q(\cdot)}(L^{p(\cdot)})$ and its dual space $ (\ell^{q(\cdot)}(L^{p(\cdot)}))^* $ are uniformly convex and uniformly smooth. Finally, in Section \ref{section5}, we conclude the paper by presenting an application of the uniform convexity of $\ell^{q(\cdot)}(L^{p(\cdot)})$.

\section{preliminaries}\label{section2}
We recall the definition of variable exponent Lebesgue spaces. Let \( \mathcal{P}_0 \) denote the collection of all measurable functions \( p(\cdot) : \mathbb{R}^n \to [c,\infty] \), where \( c>0 \). Furthermore, let \( \mathcal{P} \subset \mathcal{P}_0 \) be the subclass consisting of those exponents whose range is contained in \( [1,\infty] \). For a variable exponent \( p(\cdot) \in \mathcal{P}_0 \), we define
\begin{align*}
p_+ := \operatorname*{ess\,sup}_{x \in \mathbb{R}^n} p(x),\qquad p_- := \operatorname*{ess\,inf}_{x \in \mathbb{R}^n}p(x).
\end{align*}
Let 
\begin{align*}
\Omega_\infty := \{ x \in \mathbb{R}^n : p(x) = \infty \},\qquad\Omega_0 := \mathbb{R}^n \setminus \Omega_\infty
\end{align*}
denote the subsets of \( \mathbb{R}^n \) on which the exponent is infinite and finite, respectively. The variable exponent Lebesgue space \( L^{p(\cdot)}(\mathbb{R}^n) \) is defined as the collection of all measurable functions \( f \) for which there exists \( \lambda > 0 \) such that the modular 
\begin{align*}
\varrho_{L^{p(\cdot)}}(\lambda^{-1} f) := \int_{\Omega_0} \left( \frac{|f(x)|}{\lambda} \right)^{p(x)} \, dx + \bigl\| \lambda^{-1} f \bigr\|_{L^{\infty}(\Omega_\infty)}
\end{align*}
is finite. For \( p(\cdot) \in \mathcal{P}_0 \) and \( f \in L^{p(\cdot)}(\mathbb{R}^n) \), the Luxemburg norm associated with \( L^{p(\cdot)}(\mathbb{R}^n) \) is defined by 
\begin{align*}
\|f\|_{L^{p(\cdot)}}:= \inf \left\{ \lambda > 0 :\varrho_{L^{p(\cdot)}}(\lambda^{-1} f) \le 1 \right\}.
\end{align*}
We introduce the variable mixed Lebesgue--sequence spaces \( \ell^{q(\cdot)}(L^{p(\cdot)}(\mathbb{R}^n)) \), following the framework of \cite{AlHa2010}. Let \( p(\cdot), q(\cdot) \in \mathcal{P}_0 \), and let \( (f_\nu)_{\nu \in \mathbb{N}} \) be a sequence of measurable functions such that \( f_\nu \in L^{p(\cdot)}(\mathbb{R}^n) \) for every \( \nu \in \mathbb{N} \). The modular of the variable mixed space is defined by
\begin{align}\label{P1}
\varrho_{\ell^{q(\cdot)}(L^{p(\cdot)})}\bigl( (f_\nu)_\nu \bigr):= \sum_{\nu=1}^{\infty}\inf \left\{ \lambda_\nu > 0 :\varrho_{L^{p(\cdot)}}\!\left(\frac{f_\nu}{\lambda_\nu^{1/q(\cdot)}}\right) \le 1 \right\},
\end{align}
where we adopt the convention \( \lambda^{1/\infty} := 1 \). If \( q_+ < \infty \), the modular admits the equivalent representation
\begin{equation}\label{P1a}
\varrho_{\ell^{q(\cdot)}(L^{p(\cdot)})}\bigl( (f_\nu)_\nu \bigr)= \sum_{\nu=1}^{\infty}\bigl\| |f_\nu|^{\,q(\cdot)} \bigr\|_{L^{p(\cdot)/q(\cdot)}} .
\end{equation}
The variable mixed space \( \ell^{q(\cdot)}(L^{p(\cdot)}(\mathbb{R}^n)) \) is then defined as 
\begin{align*}
\ell^{q(\cdot)}(L^{p(\cdot)}):= \left\{ (f_\nu)_\nu :\exists\, \lambda > 0 \text{ such that }\varrho_{\ell^{q(\cdot)}(L^{p(\cdot)})}\bigl( (f_\nu)_\nu / \lambda \bigr) < \infty \right\}.
\end{align*}
For \( p(\cdot), q(\cdot) \in \mathcal{P}_0 \), the space
\( \ell^{q(\cdot)}(L^{p(\cdot)}(\mathbb{R}^n)) \) is a quasi-normed space, with quasi-norm given by
\begin{equation}\label{P1b}
\|(f_\nu)_\nu\|_{\ell^{q(\cdot)}(L^{p(\cdot)})}:= \inf \left\{ \mu > 0 :\varrho_{\ell^{q(\cdot)}(L^{p(\cdot)})}\bigl( (f_\nu)_\nu / \mu \bigr) \le 1 \right\}.
\end{equation}
As shown in \cite[Theorem~1]{KeVy} and \cite[Theorem~3.6]{AlHa2010}, this quasi-norm is in fact a norm provided that \( p(\cdot), q(\cdot) \in \mathcal{P} \) and one of the following conditions holds:
\begin{enumerate}
\item \( 1 \le q(\cdot) \le p(\cdot) \le \infty \);
\item \( p(\cdot) \ge 1 \) and \( q(\cdot) \ge 1 \) is constant;
\item \( \frac{1}{p(\cdot)} + \frac{1}{q(\cdot)} \le 1 \).
\end{enumerate} 
Throughout this paper, inequalities involving variable exponents are understood in the almost everywhere sense. All results stated below hold whenever the space \( \ell^{q(\cdot)}(L^{p(\cdot)}) \) is a Banach space.

In order to study duality properties of vector-valued sequence spaces and compactness arguments later on, we recall the definition of the K\"{o}the dual and the notion of (uniform) monotonicity.

\begin{definition}\label{defi:220130}
Let $S(X)$ be a vector-valued sequence space over a reflexive Banach space $X$. The K\"{o}the dual of $S(X)$ with respect to the dual pair $(X,X^*)$ (see \cite{GKP}) is defined by 
\begin{align*}
S(X)' := \left\{ \bar{\varphi}=(\varphi_\nu)_{\nu=1}^\infty \in (X^*)^{\mathbb{N}} : \sum_{\nu=1}^\infty \bigl| \varphi_\nu(f_\nu) \bigr| < \infty\ \text{for every } \bar{f}=(f_\nu)_{\nu=1}^\infty \in S(X) \right\}.
\end{align*}
The norm of an element $\bar{\varphi}=(\varphi_\nu)_{\nu=1}^\infty \in S(X)'$ is given by 
\begin{align*}
\|\bar{\varphi}\|_{S(X)'} :=\sup \left\{\sum_{\nu=1}^\infty \bigl| \varphi_\nu(f_\nu) \bigr| :\bar{f}=(f_\nu)_{\nu=1}^\infty \in S(X),\;\|\bar{f}\|_{S(X)} \le 1\right\}.
\end{align*}
\end{definition}
We recall that there exists a canonical isomorphism between the dual and the K\"{o}the dual of the space $\ell^{q(\cdot)}(L^{p(\cdot)})$. More precisely, when $ 1<q_-,p_- $, by \cite[Lemma~2.4]{GRS} and \cite[Definition~2.2]{GRS}, the mapping
\begin{align}\label{Kothedual}
\psi : \ell^{q(\cdot)}(L^{p(\cdot)})' &\longrightarrow (\ell^{q(\cdot)}(L^{p(\cdot)}))^*,
\\
f' &\longmapsto \psi_{f'}, \qquad \text{with } \|f'\|_{\ell^{q(\cdot)}(L^{p(\cdot)})'} = \|\psi_{f'}\|_{(\ell^{q(\cdot)}(L^{p(\cdot)}))^*}, \notag
\end{align}
is a canonical isometric isomorphism. Here $\ell^{q(\cdot)}(L^{p(\cdot)})'$ denotes the K\"othe dual of $\ell^{q(\cdot)}(L^{p(\cdot)})$, and for every $g \in \ell^{q(\cdot)}(L^{p(\cdot)})$, 
\begin{align*}
\psi_{f'}(g) = \sum_{\nu=1}^{\infty} \int_{\mathbb{R}^n} f'_\nu(x)\, g_\nu(x)\, dx, \qquad f'_\nu \in L^{p'(\cdot)},\; g_\nu \in L^{p(\cdot)},
\end{align*}
where $p'(\cdot)$ denotes the conjugate exponent of $p(\cdot)$.

\section{Strict Convexity and Smoothness}\label{section3}

The space $\ell^{q(\cdot)}(L^{p(\cdot)})$ is a mixed sequence space with a rather intricate structure, which makes the analysis nontrivial. For the reader’s convenience, we begin by recalling several notions and results that are essential for the proof.

\begin{definition}\cite[Page~101]{LP}
A Banach space $X$ is said to be \emph{strictly convex} if its unit ball $B(X)$ satisfies the following property: for any two distinct elements $x,y \in B(X)$, 
\begin{align*}
\|x+y\|_X < 2.
\end{align*}
\end{definition}

\begin{definition}\cite{LP}\label{def:Smooth}
A Banach space $X$ is called \emph{G\^ateaux smooth} (or simply \emph{smooth}) if every point $x \in S(X)$ on the unit sphere of $X$ is smooth, that is, for every $y \in X$, the limit 
\begin{align*}
\lim_{t \to 0} \frac{\|x + ty\|_X - \|x\|_X}{t}
\end{align*}
exists.
\end{definition}
\begin{theorem}\cite[Theorem 2.2.4]{LP}\label{theo:SC,Smooth}
Let $X$ be a Banach space.
\begin{itemize}
\item[(i)] If $X^*$ is strictly convex, then $X$ is smooth.
\item[(ii)] If $X^*$ is smooth, then $X$ is strictly convex.
\end{itemize}
\end{theorem}
\begin{remark}\label{J:reflexivity}
Let $1 < q_-, p_-, q_+, p_+ < \infty$. Then the space $\ell^{q(\cdot)}(L^{p(\cdot)})$ is reflexive \cite[Theorem~1.1]{GRS}. Moreover, the canonical embedding 
\begin{align*}
\Gamma : \ell^{q(\cdot)}(L^{p(\cdot)}) \longrightarrow (\ell^{q(\cdot)}(L^{p(\cdot)}))^{**}, \qquad f \longmapsto \Gamma_f,
\end{align*}
is an isometric isomorphism.
\end{remark}

\begin{theorem}\label{theo:SCandBoundedExpons}
Let $ q_->1,~ p_+<\infty $ and $ q(\cdot)\leq p(\cdot) $ almost everywhere. Then $ \ell^{q(\cdot)}(L^{p(\cdot)}) $ is strictly convex.
\end{theorem}
\begin{proof}
Let $ f,g \in S\bigl(\ell^{q(\cdot)}(L^{p(\cdot)})\bigr) $ be two distinct elements. Since $q(\cdot)\leq p(\cdot)$, we have $ 1\leq {p(\cdot)}/{q(\cdot)}\leq {p_+}/{q_-}<\infty $. Moreover, since $q_+<\infty$ we can use the representation \eqref{P1a}. For every $\nu$, since $q(x)>1$ a.e., the strict convexity of $t\mapsto |t|^{q(x)}$ gives 
\begin{align*}
\left|\frac{f_\nu(x)+g_\nu(x)}{2}\right|^{q(x)}\leq\frac{|f_\nu(x)|^{q(x)}+|g_\nu(x)|^{q(x)}}{2}
\end{align*}
for almost every $x$. Since $f\neq g$, there exists at least one index $\nu_0$ such that $ f_{\nu_0}\neq g_{\nu_0} $ on a set of positive measure. Hence the preceding inequality is strict on a set of positive measure for $\nu=\nu_0$. On the other side, by monotonicity property of the norm $L^{p(\cdot)/q(\cdot)}$ and triangle inequality, we have
\begin{align*}
\left\| |\frac{f_\nu+g_\nu}{2} |^{q(\cdot)}\right\|_{L^{\frac{p(\cdot)}{q(\cdot)}}} \leq \left\|\frac{|f_\nu|^{q(\cdot)}+|g_\nu|^{q(\cdot)}}{2}\right\|_{L^{\frac{p(\cdot)}{q(\cdot)}}}\leq\frac{\bigl\| f_\nu^{q(\cdot)}\bigr\|_{L^{p(\cdot)/q(\cdot)}}+\bigl\|g_\nu^{q(\cdot)}\bigr\|_{L^{p(\cdot)/q(\cdot)}}}{2}.
\end{align*}
Considering strict inequality for $\nu=\nu_0$, we obtain
\begin{align*}
\varrho_{\ell^{q(\cdot)}(L^{p(\cdot)})}\left(\frac{f+g}{2}\right)<\frac{1}{2}\sum_{\nu=1}^{\infty}\Bigl(\|f_{\nu}^{q(\cdot)}\|_{\frac{p(\cdot)}{q(\cdot)}} +\|g_{\nu}^{q(\cdot)}\|_{\frac{p(\cdot)}{q(\cdot)}} \Bigr)=\frac{\varrho_{\ell^{q(\cdot)}(L^{p(\cdot)})}(f)+\varrho_{\ell^{q(\cdot)}(L^{p(\cdot)})}(g)}{2}\leq 1.
\end{align*}
Thus, by continuity of $\lambda\mapsto\varrho_{\ell^{q(\cdot)}(L^{p(\cdot)})}((f+g)/{\lambda})$ with respect to $ \lambda\in(0,1] $, 
\begin{align*}
\big\|\frac{f+g}{2}\big\|_{\ell^{q(\cdot)}(L^{p(\cdot)})}<1,
\end{align*}
which is the desired result.
\end{proof}

\begin{theorem}\label{theo:smoothX}
Let $ q_->1,~p_+<\infty $ and $ q(\cdot)\le p(\cdot) $ almost everywhere. Then \(\ell^{q(\cdot)}(L^{p(\cdot)})\) is smooth.
\end{theorem}
\begin{proof}
Note that, $ 1\le ({p(\cdot)}/{q(\cdot)})_-\leq ({p(\cdot)}/{q(\cdot)})_+\leq\frac{p_+}{q_-}<\infty $. Let $ \phi_\nu(f_\nu):=\||f_\nu|^{q(\cdot)}\|_{L^{p(\cdot)/q(\cdot)}} $ for \(\nu\in\mathbb{N}\), hence by \eqref{P1a}, 
\begin{align*}
\varrho_{\ell^{q(\cdot)}(L^{p(\cdot)})}(f) = \sum_{\nu=1}^{\infty}\phi_\nu(f_\nu).
\end{align*}
Let \(f,h\in \ell^{q(\cdot)}(L^{p(\cdot)})\). We have teo following cases,
\begin{itemize}
\item[1.] Let $ f_\nu\neq 0 $: then by differentiating
\begin{align*}
\varrho_{L^\frac{p(\cdot)}{q(\cdot)}}(\frac{|f_\nu+th_\nu|^{q(\cdot)}}{\phi_\nu(f_\nu+th_\nu)})=\int_{\mathbb{R}^n} |f_\nu(x)+t h_\nu(x)|^{p(x)} \phi_\nu(f_\nu+th_\nu)^{-\frac{p(\cdot)}{q(\cdot)}}\,dx =1,
\end{align*}
at \(t=0\), we obtain 
\begin{align*}
0=&\int_{\mathbb{R}^n} p(x)|f_\nu|^{p(x)-2}f_\nu h_\nu\phi_\nu(f_\nu)^{-\frac{p(\cdot)}{q(\cdot)}}\,dx
\\
&-\phi_\nu'(f_\nu;h_\nu)\int_{\mathbb{R}^n} \frac{p(x)}{q(x)}|f_\nu|^{p(x)}\phi_\nu(f_\nu)^{-\frac{p(\cdot)}{q(\cdot)}-1}\,dx.
\end{align*}
Thus, 
\begin{align}\label{eq:2}
\phi_\nu'(f_\nu;h_\nu)=\frac{\displaystyle\int_{\mathbb{R}^n} p(x)|f_\nu|^{p(x)-2}f_\nu h_\nu \phi_\nu(f_\nu)^{-\frac{p(\cdot)}{q(\cdot)}}\,dx}{\displaystyle\int_{\mathbb{R}^n}\frac{p(x)}{q(x)}|f_\nu|^{p(x)}\phi_\nu(f_\nu)^{-\frac{p(\cdot)}{q(\cdot)}-1}\,dx},
\end{align}
where $ \phi_\nu'(f_\nu;h_\nu) $ is the derivative of $ \phi_\nu $ at $ f_\nu $ in direction of $ h_\nu $, which is linear in \(h_\nu\) and is well-defined, since \(f_\nu\neq 0\) .
\item[2.] Let \(f_\nu=0\): then for \(|t|\le 1\) and \(q_->1\), 
\begin{align*}
\phi_\nu'(0;h_\nu)=\lim_{t\to 0}\frac{\phi_\nu(t h_\nu)-\phi_\nu(0)}{t}\leq\lim_{t\to 0}{t^{q_--1}\|h_\nu\|_{L^{p(\cdot)/q(\cdot)}}}=0.
\end{align*}
\end{itemize}
Since $ \phi_\nu(f_\nu+t h_\nu)$ is convex when $t>0 $, for \(0<|t|\le t_0\),
\begin{align}\label{eq:3}
|\frac{\phi_\nu(f_\nu+t h_\nu)-\phi_\nu(f_\nu)}{t}|\le\frac{\phi_\nu(f_\nu+t_0 h_\nu)+\phi_\nu(f_\nu-t_0 h_\nu)+2\phi_\nu(f_\nu)}{t_0}=:G_\nu.
\end{align}
On the other side, we have
\begin{align*}
\varrho_{\ell^{q(\cdot)}(L^{p(\cdot)})}'(f;h)&=\lim_{t\to 0}\frac{\varrho_{\ell^{q(\cdot)}(L^{p(\cdot)})}(f+th)-\varrho_{\ell^{q(\cdot)}(L^{p(\cdot)})}(f)}{t}
\\
&=\lim_{t\to 0}\sum_{\nu=1}^{\infty}\frac{\phi_\nu(f_\nu+t h_\nu)-\phi_\nu(f_\nu)}{t}
\end{align*}
Moreover, for $ (G_\nu)_\nu $, we have
\begin{align*}
\|(G_\nu)_\nu\|_{\ell^1}=\frac{1}{t_0}\Big(\varrho_{\ell^{q(\cdot)}(L^{p(\cdot)})}(f+t_0h)+\varrho_{\ell^{q(\cdot)}(L^{p(\cdot)})}(f-t_0h)+2\varrho_{\ell^{q(\cdot)}(L^{p(\cdot)})}(f)\Big)<\infty.
\end{align*}
Therefore, there exists a sequence $ (G_\nu)_\nu\in\ell^1 $, such that \eqref{eq:3} holds and for every $ \nu $ we have
\begin{align*}
\lim_{t\to 0}\frac{\phi_\nu(f_\nu+t h_\nu)-\phi_\nu(f_\nu)}{t}=\phi_\nu'(f_\nu;h_\nu).
\end{align*}
Thus by dominated convergence theorem for the counting measure $ \mu $,
\begin{align*}
\varrho_{\ell^{q(\cdot)}(L^{p(\cdot)})}'(f;h)&=\lim_{t\to 0}\sum_{\nu=1}^{\infty}\frac{\phi_\nu(f_\nu+t h_\nu)-\phi_\nu(f_\nu)}{t}
\\
&=\lim_{t\to 0}\int_{\mathbb{N}}\frac{\phi_\nu(f_\nu+t h_\nu)-\phi_\nu(f_\nu)}{t}d\mu
\\
&=\int_{\mathbb{N}}\lim_{t\to 0}\frac{\phi_\nu(f_\nu+t h_\nu)-\phi_\nu(f_\nu)}{t}d\mu
\\
&=\sum_{\nu=1}^{\infty} \phi_\nu'(f_\nu;h_\nu).
\end{align*}
Thus $ \varrho_{\ell^{q(\cdot)}(L^{p(\cdot)})}'(f;h) $ is a linear functional of \(h\). 

Let $ f\in S(\ell^{q(\cdot)}(L^{p(\cdot)})), h\in \ell^{q(\cdot)}(L^{p(\cdot)}) $, and  $ C_f:=\varrho_{\ell^{q(\cdot)}(L^{p(\cdot)})}'(f;f) $. Note that, by \eqref{eq:2} for $ f\neq 0 $, we have $ C_f>0 $, then we can set $ \ell:={\varrho_{\ell^{q(\cdot)}(L^{p(\cdot)})}'(f;h)}/{C_f}. $ We will consider following cases,
\begin{itemize}
\item For \(c>\ell\): and \(t>0\), sufficiently small, set $ \mu_t:=1+ct $. By G\^ateaux differentiability of modular $ \varrho_{\ell^{q(\cdot)}(L^{p(\cdot)})} $ at \(f\), and identity $ t/\mu_t\sim t $ as $ t\to 0 $, we have 
\begin{align*}
\varrho_{\ell^{q(\cdot)}(L^{p(\cdot)})}\left(\frac{f+t h}{\mu_t}\right)&=\varrho_{\ell^{q(\cdot)}(L^{p(\cdot)})}\left(f+\frac{t}{\mu_t}(h-cf)\right)
\\
&=\varrho_{\ell^{q(\cdot)}(L^{p(\cdot)})}\left(f\right)+\frac{t}{\mu_t}\varrho_{\ell^{q(\cdot)}(L^{p(\cdot)})}'(f;h-cf)+o(\frac{t}{\mu_t})
\\
&=1+\frac{t}{\mu_t}\varrho_{\ell^{q(\cdot)}(L^{p(\cdot)})}'(f;h-cf)+o(t)
\\
&=1+\frac{t}{\mu_t}\bigl(\varrho_{\ell^{q(\cdot)}(L^{p(\cdot)})}'(f;h)-c~C_f\bigr)+o(t)
\\
&=1+\frac{t}{\mu_t}\bigl(\ell C_f-c~C_f\bigr)+o(t)
\\
&< 1+o(t).
\end{align*}
Thus, $ \varrho_{\ell^{q(\cdot)}(L^{p(\cdot)})}(\frac{f+t h}{\mu_t})<1 $ which is $ \|f+t h\|_{\ell^{q(\cdot)}(L^{p(\cdot)})}\le 1+ct $. Therefore by passing $ c\to \ell $ we have 
\begin{align}\label{eq:6}
\limsup_{t\to 0^+}\frac{\|f+t h\|_X-1}{t}\leq \ell.
\end{align}
\item If \(c<\ell\), the same argument gives $ \varrho_{\ell^{q(\cdot)}(L^{p(\cdot)})}(\frac{f+t h}{\mu_t})<1 $, for sufficiently small \(t>0\) which is $ \|f+t h\|_X\ge 1+ct $, and by passing $ c\to \ell $, we obtain
\begin{align}\label{eq:7}
\liminf_{t\to 0^+} \frac{\|f+t h\|_X-1}{t} \ge \ell.
\end{align}
\end{itemize}
By \eqref{eq:6} and \eqref{eq:7}, we obtain 
\begin{align*}
\lim_{t\to 0^+}\frac{\|f+t h\|_X-1}{t}=\ell=\frac{\varrho_{\ell^{q(\cdot)}(L^{p(\cdot)})}'(f;h)}{\varrho_{\ell^{q(\cdot)}(L^{p(\cdot)})}'(f;f)}.
\end{align*}
We have the same argument for \(-h\). Hence 
\begin{align*}
\lim_{t\to 0}\frac{\|f+t h\|_{\ell^{q(\cdot)}(L^{p(\cdot)})}-\|f\|_{\ell^{q(\cdot)}(L^{p(\cdot)})}}{t}=\frac{\varrho_{\ell^{q(\cdot)}(L^{p(\cdot)})}'(f;h)}{\varrho_{\ell^{q(\cdot)}(L^{p(\cdot)})}'(f;f)},
\end{align*}
which is linear in \(h\). Therefore the norm of \(\ell^{q(\cdot)}(L^{p(\cdot)})\) is G\^ateaux differentiable at $ f\in S(\ell^{q(\cdot)}(L^{p(\cdot)})) $, and this complete the proof.
\end{proof} 
The following corollary is a direct result of Theorems \ref{theo:SCandBoundedExpons} and \ref{theo:smoothX}.
\begin{corollary}\label{coro:X^*UC}
Let $ q_->1,~ p_+<\infty $ and $ q(\cdot)\leq p(\cdot) $ almost everywhere. Then $ \ell^{q(\cdot)}(L^{p(\cdot)}) $ and its dual are strictly convex and smooth.
\end{corollary}
\begin{proof}
Since $ 1<q_-,p_-,p_+,q_+<\infty $, by \cite[Theorem 1.1]{GRS} $ \ell^{q(\cdot)}(L^{p(\cdot)}) $ is reflexive, i.e. $ \ell^{q(\cdot)}(L^{p(\cdot)})\approx\big(\ell^{q(\cdot)}(L^{p(\cdot)})\big)^{**} $. Therefore, by Theorems \ref{theo:SCandBoundedExpons}, \ref{theo:smoothX} and \ref{theo:SC,Smooth} we have the desired result.
\end{proof}

\section{Uniform Convexity and Uniform Smoothness}\label{section4}
Let's review the concepts associated with uniform convexity.

\begin{definition}\label{defUC}
A Banach space $X$ is said to be uniformly convex if for all $\varepsilon\in(0,2]$ there exists $\delta>0$ such that 
\begin{align*}
\|\frac{x+y}{2}\|\le 1-\delta
\end{align*}
whenever $ x,y \in S(X) $ satisfies $\|x-y\|\ge \varepsilon$.
\end{definition}
\begin{definition}
A normed vector space X is said to be uniformly smooth if for every $ \varepsilon>0 $ there exists $ \delta>0 $ such that if $ x,y\in X $ with $ \|x\|_{X}=1 $ and $ \|y\|_{X}\leq\delta $ then 
\begin{align*}
\|x+y\|_{X}+\|x-y\|_{X}\leq 2+\varepsilon\|y\|_{X}.
\end{align*}
\end{definition}
Recall that a Banach space $ X $ is uniformly smooth if for any $x$ in unit sphere $ S(X) $ and $ y $ in $ X $, the limit
\begin{align}\label{T2}
\lim_{t\to 0}\Phi_{x,y}(t):=\lim\limits_{t \to 0} \frac{\|x-ty\|_{X}-\|x\|_{X}}{t},
\end{align}
exists and is uniform in $x, y \in S(X)$  \cite{LP}.
\begin{definition}\label{def:UM}
A Banach lattice \(E\) is said to be uniformly monotone if, for every \(0<\varepsilon<1\), there exists \(\delta(\varepsilon)>0\) such that 
\begin{align*}
0\le y\le x,\qquad\|x\|_E=1,\qquad\|y\|_E\ge\varepsilon
\end{align*}
imply 
\begin{align*}
\|x-y\|_E\le 1-\delta(\varepsilon).
\end{align*}
See, e.g., \cite{HUDZIK20093407}.
\end{definition}
Note that, equivalent to Definition \ref{def:UM}, A Banach lattice \(E\) is said to be uniformly monotone if, for every \(0<\eta<1\), there exists \(\omega_E(\eta)>0\) such that 
\begin{align*}
0\le u\le v,\qquad\|v\|_E\le 1,\qquad\|u-v\|_E\ge\eta
\end{align*}
imply 
\begin{align*}
\|u\|_E\le 1-\omega_E(\eta).
\end{align*} 
Here, we examine if the variable mixed Lebesgue-sequence space $ \ell^{q(\cdot)}(L^{p(\cdot)}) $ is uniformly convex. We find that uniform convexity is not always a feature of this space, especially when $ q=\infty $ or $ q=1 $. In the following proposition, we will see that the question of whether the space is uniformly convex can have either a positive or negative answer, depending on specific cases.
\begin{proposition}\label{propNotUC}
Let $ p(\cdot)\in\mathcal{P} $. The following statements are true.
\begin{itemize}
\item[1.] If $q \in (1,\infty)$ is a constant and $1<p_{-} \le p_{+} < +\infty$, then the space ${\ell}^{q}(L^{p(\cdot)})$ is uniformly convex.
\item[2.] The space $\ell^{\infty}(L^{p(\cdot)})$ is not uniformly convex.
\item[3.] The space $ \ell^1(L^{p(\cdot)}) $ is not uniformly convex.
\end{itemize}
\end{proposition}
\begin{proof}
We prove each statement separately. For the endpoint examples, choose a measurable set $A\subset\mathbb R^n$ with $0<|A|<\infty$ and put $h=\chi_A/\|\chi_A\|_{L^{p(\cdot)}}$. Then $\|h\|_{L^{p(\cdot)}}=1$.
\begin{itemize}
\item[1.] By \cite[Proposition 3.3]{AlHa2010}, we have
\begin{align*}
\|(f_{\mu})_{\mu}\|_{\ell^{q(\cdot)}(L^{p(\cdot)})}=\|\|f_{\mu}\|_{L^{p(\cdot)}}\|_{\ell^q}.
\end{align*}
Since $L^{p(\cdot)}$ is uniformly convex, by \cite[Theorem 3]{Day} it follows that ${\ell}^{q}(L^{p(\cdot)}(\mathbb R^n))$ is also uniformly convex.

\item[2.]
To show that $\ell^{\infty}(L^{p(\cdot)})$ is not uniformly convex, it suffices to consider the functions $f=(h, h, \cdots )$ and $g=(0, h,0,h, \cdots )$.  Then we have
\begin{align*}
\|f\|_{\ell^{\infty}(L^{p(\cdot)})} &= \|g\|_{\ell^{\infty}(L^{p(\cdot)})}=1
\\
\|f-g\|_{\ell^{\infty}(L^{p(\cdot)})}=1 &\quad\mbox{and}\quad \|\frac{f+g}{2}\|_{\ell^{\infty}(L^{p(\cdot)})}=1,
\end{align*}
which is the desired result.
\item[3.] Similarly, to show that $ \ell^1(L^{p(\cdot)}) $ is not uniformly convex, we can consider the functions $f=(h, 0,0, \cdots )$ and $g=( 0,h,0,0, \cdots )$. Then we have
\begin{align*}
\|f\|_{\ell^{1}(L^{p(\cdot)})} &= \|g\|_{\ell^{1}(L^{p(\cdot)})}=1,
\\
\quad \|f-g\|_{\ell^{1}(L^{p(\cdot)})}=2 &\quad\mbox{and}\quad\|\frac{f+g}{2}\|_{\ell^{1}(L^{p(\cdot)})}=1
\end{align*}
which also violates the definition of uniform convexity.
\end{itemize}
This completes the proof.
\end{proof} 
In general, it is not enough to guarantee that the modular $ \varrho_{\ell^{q(\cdot)}(L^{p(\cdot)})}(\cdot) $ is convex when $ p,q\geq 1 $ \cite{AlHa2010}, and the convexity of the modular is sufficient but not necessary for the convexity of the norm. Here we investigate uniform convexity property of Banach spaces, especially, the modular  $ \varrho_{\ell^{q(\cdot)}(L^{p(\cdot)})}(\cdot) $ and the norm of the variable exponent Lebesgue-sequence space $ \ell^{q(\cdot)}(L^{p(\cdot)}) $ as a Banach space.
 
The following theorem helps to examine the uniformly convex property for the norm of the space defined in \eqref{P1b}, when the exponent $q$ is variable.

\begin{theorem}\label{theo:T11}
\cite[Theorem 2.2.5]{LP} For every Banach space, $ X $ is uniformly convex if and only if its dual space $ X^* $ is uniformly smooth.
\end{theorem}
\begin{definition}
Let $X$ be a Banach space. The (normalized) duality mapping
$J : S(X) \to 2^{S(X^*)}$ is defined by
\[
J(x) := \{ x^* \in S(X^*) : \langle x^*, x \rangle = 1 \}.
\]
\end{definition}
By the Hahn--Banach theorem, the set $J(x)$ is nonempty for every $x \in S(X)$. The space $X$ is G\^ateaux smooth if and only if $J(x)$ is a singleton for every $x\in S(X)$. The derivative of the norm at $x$ in the direction $-y$ is given by 
\begin{align*}
\lim_{t\to 0^+}\frac{\|x-ty\|_{X}-\|x\|_{X}}{t}=-\langle J (x),y\rangle,
\end{align*}
and for every $x\in S(X)$ there exists a unique $x^{*}\in S(X^{*})$ such that $ \langle x^{*},x\rangle=1 $. Thus, the duality mapping $ J $ is single-valued.

\begin{lemma}\label{lem:continuity-limit}
Let $X$ be a Banach space and suppose that $X$ is G\^ateaux smooth. Then 
\begin{align*}
l:S(X)\times S(X)\longrightarrow\mathbb{R}
\end{align*}
is continuous.
\end{lemma}
\begin{proof}
Since $X$ is G\^ateaux smooth, 
\begin{align*}
l(x,y):=\lim_{t\to 0^+}\Phi_{x,y}(t)
\end{align*}
exists and the normalized duality mapping $ J $ is single-valued; that is, for every $x\in S(X)$, there exists a unique $J(x)\in S(X^{*})$ such that $ \langle J(x),x\rangle=1 $.

By the definition of the G\^ateaux derivative of the norm, for every $(x,y)\in S(X)\times S(X)$,
\begin{align}\label{eq:representation of l}
l(x,y)=\lim_{t\to 0^+}\frac{\|x-ty\|_{X}-\|x\|_{X}}{t} =-\langle J(x),y\rangle.
\end{align}
We first show that 
\begin{align*}
J(x_n)\overset{w^*}{\longrightarrow}J(x)\qquad\text{whenever}\qquad x_n\to x\quad\text{in }X.
\end{align*}
Since ${J(x_n)}\subset B_{X^*}$, by Banach--Alaoglu theorem, every subsequence has a further convergent subsequence. Let
\begin{align}\label{eq:weak*-Convergence}
J(x_{n_k}) \overset{w^*}{\longrightarrow} z^{*}
\end{align}
be any weak-convergent subsequence. Then 
\begin{align}\label{eq:normofz*}
\|z^{*}\|_{X^{*}}\leq 1, 
\end{align}
and $ \langle J(x_{n_k}),x_{n_k}\rangle=1 $. Hence
\begin{align*}
1 =|\langle J(x_{n_k}),x\rangle+\langle J(x_{n_k}),x_{n_k}-x\rangle| \leq |\langle J(x_{n_k}),x\rangle|+\| J(x_{n_k})\|_{X^{*}}~\|x_{n_k}-x\|_{X},
\end{align*}
so $ \langle J(x_{n_k}),x\rangle \longrightarrow 1 $. Thus, by \eqref{eq:weak*-Convergence} we have $ \langle z^{*},x\rangle=1 $, together with \eqref{eq:normofz*} and  $ \|x\|_{X}=1 $, we conclude that $ z^{*} $ is a norming functional for $x$ and by uniqueness  
\begin{align*}
z^{*}=J(x).
\end{align*}
Since every weak cluster point of
$\{J(x_n)\}_n$ is equal to $J(x)$, it follows that
\begin{align}\label{eq:WeakConvergenceofJ_X}
J(x_n)\overset{w^*}{\longrightarrow}J(x).
\end{align}
Let $ x_n\to x $ and $ y_n\to y $ in $ X $. Using the representation as in \eqref{eq:representation of l}, we have
\begin{align*}
\left|l(x_n,y_n)-l(x,y)\right|&=\left|\langle J(x_n),y_n\rangle-\langle J(x),y\rangle\right|
\\
&\leq\left|\langle J(x_n),y_n-y\rangle\right|+\left|\langle J(x_n)-J(x),y\rangle\right|
\\
&\leq\|y_n-y\|_{X}+\left|\langle J(x_n)-J(x),y\rangle\right|.
\end{align*}
By \eqref{eq:WeakConvergenceofJ_X} and the assumption $y_n\to y$ in $ X $, we have the desired result.
\end{proof}
In the following theorem, another sufficient condition for uniformly convex spaces is presented.
\begin{theorem}\label{theo:UCequiv2}
Let $ X $ be a normed linear space.  If  $ X $ is smooth, and $ D $ is a dense uniformly smooth linear subspace of \(X\), equipped with the norm inherited from $ X $. Then $ X^* $ is uniformly convex.
\end{theorem}
\begin{proof}
By smoothness of $ X $, for every $ x\in S(X) $ and every $ y\in X $, there exists a $ l(x,y) $ such that
\begin{align*}
\lim_{t\rightarrow 0}\Phi_{x,y}(t)=l(x,y).
\end{align*} 
Since $D$ is dense in $X$, the set $ \big(D\setminus\{0\}\big)^2=D\setminus\{0\}\times D\setminus\{0\} $ is dense in $ \big(X\setminus\{0\}\big)^2=X\setminus\{0\}\times X\setminus\{0\} $. Moreover, for 
\begin{align*}
F_t(f,g):=\left|\Phi_{\frac{f}{\|f\|_{X}},\,\frac{g}{\|g\|_{X}}}(t)-l(\frac{f}{\|f\|_{X}},\,\frac{g}{\|g\|_{X}})\right|,
\end{align*}
by Lemma \ref{lem:continuity-limit} and  continuity of $ \Phi_{f,g}(t) $ on $ X\times X\times (0,\infty) $, $F_t(f,g)$ is continuous. Therefore,
\begin{align*}
\sup_{(x,y)\in\big( S(X)\big)^2}|\Phi_{x,y}(t)-l(x,y)|&=\sup_{(x,y)\in \big(X\setminus\{0\}\big)^2}F_t(x,y)
\\
&=\sup_{(u,v)\in \big(D\setminus\{0\} \big)^2}F_t(u,v)
\\
&=\sup_{(u,v)\in \big(S(D)\big)^2}|\Phi_{u,v}(t)-l(u,v)|.
\end{align*}
Consequently, by uniform smoothness property of $ D $, right-hand side converges to zero as $t\to 0$, then so does the left-hand side. Thus by Theorem \ref{theo:T11} we have the desired result.
\end{proof}

\begin{lemma}\label{lem:thetaInequality}
Let $ 1<q_-\le q\le q_+<\infty $ and $ \theta:=\min\left\{1,\frac{q_-}{2}\right\} $. Then there exists a constant \(C>0\), depending only on \(q_-\) and \(q_+\), such that, 
\begin{align*}
|a-b|^q\leq C\left( \frac{|a|^q+|b|^q}{2} - \left|\frac{a+b}{2}\right|^q \right)^{\theta} \left( \frac{|a|^q+|b|^q}{2} \right)^{1-\theta},
\end{align*}
for all \(a,b\in\mathbb{R}\).
\end{lemma} 
\begin{proof}
Let $ \theta_q:=\min\left\{1,\frac{q}{2}\right\} $. We distinguish two cases. 
\begin{itemize}
\item[1.] Let \(1<q<2\): By \cite[Lemma 1]{math8060870}, we have
\begin{align*}
\frac{|a|^q+|b|^q}{2} - \left|\frac{a+b}{2}\right|^q \geq \frac{q(q-1)}{2^{q+1}} \frac{|a-b|^2}{(|a|+|b|)^{2-q}}.
\end{align*}
Thus
\begin{align*}
|a-b|^q \leq C_1 \left( \frac{|a|^q+|b|^q}{2} - \left|\frac{a+b}{2}\right|^q \right)^{q/2} \left( \frac{|a|^q+|b|^q}{2} \right)^{1-q/2},
\end{align*}
where \(C_1>0\) depends only on \(q_-\) and \(q_+\).

\item[2.] Let \(q\geq 2\): By Clarkson's inequality,
\begin{align*}
|a-b|^q \leq 2^q \left( \frac{|a|^q+|b|^q}{2} - \left|\frac{a+b}{2}\right|^q \right).
\end{align*}
Since in this case \(\theta_q=1\), the desired estimate follows.
\end{itemize}
Therefore, there exists a constant \(C>0\), depending only on \(q_-\) and \(q_+\), such that
\begin{align*}
|a-b|^q &\leq C \left( \frac{|a|^q+|b|^q}{2} - \left|\frac{a+b}{2}\right|^q \right)^{\theta_q} \left( \frac{|a|^q+|b|^q}{2} \right)^{1-\theta_q}.
\end{align*}
Finally, since $ \theta_q\geq\theta $ and 
\[
0\leq \frac{|a|^q+|b|^q}{2} - \left|\frac{a+b}{2}\right|^q \leq \frac{|a|^q+|b|^q}{2},
\]
we obtain
\begin{align*}
&\left( \frac{|a|^q+|b|^q}{2} - \left|\frac{a+b}{2}\right|^q \right)^{\theta_q} \left( \frac{|a|^q+|b|^q}{2} \right)^{1-\theta_q}
\\
&\qquad\leq \left( \frac{|a|^q+|b|^q}{2} - \left|\frac{a+b}{2}\right|^q \right)^{\theta} \left( \frac{|a|^q+|b|^q}{2} \right)^{1-\theta},
\end{align*}
which proves the result.
\end{proof}

\begin{lemma}\label{lem:UCmodular}
Let $ q_->1,~p_+<\infty $ and $ q(\cdot)\le p(\cdot) $ almost everywhere. Then, for every \(0<\varepsilon\leq 2\), there exists \(\alpha(\varepsilon)>0\), depending only on \(\varepsilon,q_-,q_+\) and \(p_+\), such that 
\begin{align*}
\|f\|_{\ell^{q(\cdot)}(L^{p(\cdot)})}=\|g\|_{\ell^{q(\cdot)}(L^{p(\cdot)})}=1,\qquad \|f-g\|_{\ell^{q(\cdot)}(L^{p(\cdot)})}\ge\varepsilon
\end{align*}
imply 
\begin{align*}
\varrho_{\ell^{q(\cdot)}(L^{p(\cdot)})}\left(\frac{f+g}{2}\right) \le 1-\alpha(\varepsilon).
\end{align*}
\end{lemma}

\begin{proof}
Let $ \varepsilon>0 $, $ 1\le ({p}/{q})_-\leq ({p}/{q})_+<\infty $, \(f,g\in S(\ell^{q(\cdot)}(L^{p(\cdot)}))\) and $ \|f-g\|_{\ell^{q(\cdot)}(L^{p(\cdot)})}\ge\varepsilon $. By assuming $ \Phi(x,t)=t^{p(x)/q(x)}, $ \cite[Theorem 2.7]{KURC1992173} yields that \(  L^{p(\cdot)/q(\cdot)}\), is uniformly monotone. Hence, there exists a modulus of monotonicity \(\omega_{L^{p(\cdot)/q(\cdot)}}\) such that $ \omega_{L^{p(\cdot)/q(\cdot)}}(\varepsilon)>0 $. On the other hand, \(\ell^1\) is uniformly monotone. Indeed, for \(x,y\in(\ell^1)_+\) satisfying 
\[
0\le y\le x,\qquad \|x\|_{\ell^1}=1,\qquad \|y\|_{\ell^1}\ge\varepsilon,
\]
we have 
\[
\|x-y\|_{\ell^1} = \|x\|_{\ell^1}-\|y\|_{\ell^1} \le 1-\varepsilon.
\] 
Thus, $ \omega_{\ell^1}(\varepsilon)=\varepsilon $ can be chosen as the modulus of monotonicity of \(\ell^1\). Therefore, by \cite[Theorem 3]{HUDZIK2009459}, $ \ell^1\bigl(L^{p(\cdot)/q(\cdot)}\bigr) $ is uniformly monotone. For every \(\nu\), define 
\begin{align*}
M_\nu:=\frac{|f_\nu|^{q(\cdot)}+|g_\nu|^{q(\cdot)}}{2}, \qquad H_\nu:= \left|\frac{f_\nu+g_\nu}{2}\right|^{q(\cdot)}, \qquad \Delta_\nu:=M_\nu-H_\nu.
\end{align*}
By the convexity of \(z\mapsto |z|^q\) for \(q>1\), $ 0\le H_\nu\le M_\nu $ a.e. for every $ \nu $, i.e.
\begin{align}\label{eq:1'}
0\leq H\leq M \quad\text{in}~\ell^1(L^{p(\cdot)/q(\cdot)}),
\end{align}
for $ M:=(M_\nu)_\nu $ and $ H:=(H_\nu)_\nu $. Since $ f,g\in S(\ell^{q(\cdot)}(L^{p(\cdot)})) $, we have
\begin{align}\label{eq:2'}
\|M\|_{\ell^1(L^{p(\cdot)/q(\cdot)})} \le \frac{\varrho_{\ell^{q(\cdot)}(L^{p(\cdot)})}(f)+\varrho_{\ell^{q(\cdot)}(L^{p(\cdot)})}(g)}{2} \le 1.
\end{align}
Let $ \varepsilon_0:=\min\{\varepsilon,1\} $, then $ \|f-g\|_{\ell^{q(\cdot)}(L^{p(\cdot)})}\geq\varepsilon $  implies
\begin{align}\label{eq:ModularInequality}
\varrho_{\ell^{q(\cdot)}(L^{p(\cdot)})}(f-g)\geq\varepsilon_0^{q_+}.
\end{align}
Set $ \theta:=\min\left\{1,\frac{q_-}{2}\right\} $ and \(q\in[q_-,q_+]\). Applying Lemma \ref{lem:thetaInequality} point wise with \(q=q(x)\), \(a=f_\nu(x)\) and \(b=g_\nu(x)\) we obtain 
\begin{align*}
|f_\nu-g_\nu|^{q(\cdot)} \le C\left(\frac{|f_\nu|^{q(\cdot)}+|g_\nu|^{q(\cdot)}}{2} -  |\frac{f_\nu+g_\nu}{2} |^{q(\cdot)}\right)^{\theta} \left(\frac{|f_\nu|^{q(\cdot)}+|g_\nu|^{q(\cdot)}}{2}\right)^{1-\theta}.
\end{align*}
By the inequality,  $ \|u^\theta v^{1-\theta}\|_{L^{p(\cdot)/q(\cdot)}} \le \|u\|_{L^{p(\cdot)/q(\cdot)}}^\theta \|v\|_{L^{p(\cdot)/q(\cdot)}}^{1-\theta} $ in $ L^{p(\cdot)/q(\cdot)} $, we have
\begin{align*}
\varrho_{\ell^{q(\cdot)}(L^{p(\cdot)})}(f-g)=\sum_{\nu=1}^{\infty} \bigl\||f_\nu-g_\nu|^{q(\cdot)} \bigr\|_{L^{p(\cdot)/q(\cdot)}}\le C \sum_{\nu=1}^{\infty} \|\Delta_\nu\|_{L^{p(\cdot)/q(\cdot)}}^\theta \|M_\nu\|_{L^{p(\cdot)/q(\cdot)}}^{1-\theta}.
\end{align*}
\begin{itemize}
\item[1.] If \(\theta<1\): H\"{o}lder's inequality for sequences gives  
\begin{align*}
\varrho_{\ell^{q(\cdot)}(L^{p(\cdot)})}(f-g) &\le C \left( \sum_{\nu=1}^{\infty} \|\Delta_\nu\|_{L^{p(\cdot)/q(\cdot)}} \right)^\theta \left( \sum_{\nu=1}^{\infty} \|M_\nu\|_{L^{p(\cdot)/q(\cdot)}} \right)^{1-\theta}
\\
&= C\|\Delta\|_{\ell^1(L^{p(\cdot)/q(\cdot)})}^\theta \|M\|_{\ell^1(L^{p(\cdot)/q(\cdot)})}^{1-\theta}
\\
&\le C\|\Delta\|_{\ell^1(L^{p(\cdot)/q(\cdot)})}^\theta.
\end{align*}
\item[2.] If \(\theta=1\): then \(q_-\ge 2\), hence \(q(x)\ge 2\) a.e. and
\begin{align*}
\varrho_{\ell^{q(\cdot)}(L^{p(\cdot)})}(f-g)\le C\sum_{\nu=1}^{\infty} \|\Delta_\nu\|_{L^{p(\cdot)/q(\cdot)}} = C\|\Delta\|_{\ell^1(L^{p(\cdot)/q(\cdot)})}^\theta.
\end{align*}
\end{itemize}
Thus in both cases, 
\begin{align*}
\varrho_{\ell^{q(\cdot)}(L^{p(\cdot)})}(f-g)\le C\|\Delta\|_{\ell^1(L^{p(\cdot)/q(\cdot)})}^\theta.
\end{align*}
Combining this with \eqref{eq:ModularInequality}, we obtain
\begin{align}\label{eq:3'}
\|M-H\|_{\ell^1(L^{p(\cdot)/q(\cdot)})}=\sum_{\nu}\|\Delta_\nu\|_{L^{p(\cdot)/q(\cdot)}}\geq c\,\varepsilon_0^{q_+/\theta}=:\eta_\varepsilon
\end{align}
where \(c>0\) depends only on \(q_-\), \(q_+\), and \(p_+\). Hence, by \eqref{eq:1'}, \eqref{eq:2'}, \eqref{eq:3'} and the uniform monotonicity of \(\ell^1(L^{p(\cdot)/q(\cdot)})\), there exists a modulus of monotonicity \(\omega_{\ell^1(L^{p(\cdot)/q(\cdot)})}\) such that
\begin{align*}
\varrho_{\ell^{q(\cdot)}(L^{p(\cdot)})}\left(\frac{f+g}{2}\right)=\left\| H\right\|_{\ell^1(L^{p(\cdot)/q(\cdot)})} \leq 1-\omega_{\ell^1(L^{p(\cdot)/q(\cdot)})}(\eta_\varepsilon). 
\end{align*}
Thus, setting $ \alpha(\varepsilon):=\omega_{\ell^1(L^{p(\cdot)/q(\cdot)})}(\eta_\varepsilon)>0 $, we obtain the desired result.
\end{proof}

\begin{theorem}\label{theo:XisUC}
Let \(q_->1\), \(p_+<\infty\), and $ q(\cdot)\le p(\cdot) $ almost everywhere. Then $ \ell^{q(\cdot)}(L^{p(\cdot)}) $ is uniformly convex.
\end{theorem}

\begin{proof}
Let \(f,g\in S(\ell^{q(\cdot)}(L^{p(\cdot)}))\) and let \(\varepsilon>0\) be such that 
\begin{align*}
\|f-g\|_{\ell^{q(\cdot)}(L^{p(\cdot)})}\ge\varepsilon.
\end{align*}
By Lemma~\ref{lem:UCmodular}, there exists \(\alpha(\varepsilon)>0\) such that 
\begin{align*}
\varrho_{\ell^{q(\cdot)}(L^{p(\cdot)})}\left(\frac{f+g}{2}\right) \le 1-\alpha(\varepsilon)<1.
\end{align*}
Hence $ \left\|{(f+g)}/{2}\right\|_{\ell^{q(\cdot)}(L^{p(\cdot)})}<1 $ and by norm--modular estimate we have
\begin{align*}
\|\frac{f+g}{2}\|_{\ell^{q(\cdot)}(L^{p(\cdot)})}\le(1-\alpha(\varepsilon))^{1/q_+}.
\end{align*}
Let $ \delta(\varepsilon):=1-\left(1-\alpha(\varepsilon)\right)^{1/q_+}>0 $, then 
\begin{align*}
\|\frac{f+g}{2}\|_{\ell^{q(\cdot)}(L^{p(\cdot)})}\le 1-\delta(\varepsilon),
\end{align*}
which is the desired result.
\end{proof}
\begin{definition}\label{defModUC}
Let $\varrho_X$ be a convex modular on a vector space $X$. We say that $\varrho$ is \emph{uniformly convex} if for every $r>0$ and $\varepsilon_X>0$, there exists $\delta=\delta(r,\varepsilon)>0$ such that 
\[
\varrho_X\left(\frac{f+g}{2}\right) \leq r(1-\delta)
\]
whenever $f,g\in X$ satisfy 
\[
\varrho_X(f)\leq r,\qquad \varrho_X(g)\leq r,\qquad \varrho_X(f-g)\geq\varepsilon r.
\] 
\end{definition}

\begin{proposition}\label{prop:notUC}
Let $ q_-=1 $ and $ q(\cdot)\le p(\cdot)\le p_+<\infty $. Then $ \ell^{q(\cdot)}(L^{p(\cdot)}) $ is not uniformly convex. Moreover, its modular $ \varrho_{\ell^{q(\cdot)}(L^{p(\cdot)})} $ is not uniformly convex.
\end{proposition}

\begin{proof}
Let \(\varepsilon>0\). Since \(q_-=1\), the set $ \Omega_\varepsilon:=\{x\in\Omega:q(x)<1+\varepsilon\} $ has positive measure. Let \(h_\varepsilon\) be such that $ supp(h_\varepsilon)\subset\Omega_\varepsilon $, and $ \||h_\varepsilon|^{q(\cdot)} \|_{L^{p(\cdot)/q(\cdot)}}=1 $. Define $ f_\varepsilon:=(h_\varepsilon,0,0,\ldots),~ g_\varepsilon:=(0,h_\varepsilon,0,\ldots) $. Then
\begin{align}\label{eq:ModularNOT-UC-1}
\varrho_{\ell^{q(\cdot)}(L^{p(\cdot)})}(f_\varepsilon)=\varrho_{\ell^{q(\cdot)}(L^{p(\cdot)})}(g_\varepsilon)=1,
\end{align}
and hence
\begin{align}\label{eq:NormNOT-UC-1}
\|f_\varepsilon\|_{\ell^{q(\cdot)}(L^{p(\cdot)})}=\|g_\varepsilon\|_{\ell^{q(\cdot)}(L^{p(\cdot)})}=1.
\end{align}
Moreover,
\begin{align}\label{eq:ModularNOT-UC-2}
\varrho_{\ell^{q(\cdot)}(L^{p(\cdot)})}(f_\varepsilon-g_\varepsilon)=2,
\end{align}
therefore, $ \|f_\varepsilon-g_\varepsilon\|_{\ell^{q(\cdot)}(L^{p(\cdot)})}>1 $, and since \(q_+<\infty\) by \eqref{P1a} we have
\begin{align*}
1&\geq\varrho_{\ell^{q(\cdot)}(L^{p(\cdot)})}(\frac{f_\varepsilon-g_\varepsilon}{\|f_\varepsilon-g_\varepsilon\|_{\ell^{q(\cdot)}(L^{p(\cdot)})}})
\\
&=\sum_{\nu=1}^{\infty}\||\frac{f_\nu^\varepsilon-g_\nu^\varepsilon}{\|f_\varepsilon-g_\varepsilon\|_{\ell^{q(\cdot)}(L^{p(\cdot)})}}|^{q(\cdot)}\|_{\frac{p(\cdot)}{q(\cdot)}}
\\
&\geq \frac{1}{\|f_\varepsilon-g_\varepsilon\|^{q_+}_{\ell^{q(\cdot)}(L^{p(\cdot)})}}\varrho_{\ell^{q(\cdot)}(L^{p(\cdot)})}(f_\varepsilon-g_\varepsilon)
\end{align*}
Thus,
\begin{align}\label{eq:NormNOT-UC-2}
\|f_\varepsilon-g_\varepsilon\|_{\ell^{q(\cdot)}(L^{p(\cdot)})}\ge 2^{1/q_+}.
\end{align}
On the other hand, 
\begin{align*}
\varrho_{\ell^{q(\cdot)}(L^{p(\cdot)})}\left(\frac{f_\varepsilon+g_\varepsilon}{2}\right) =2\left\|2^{-q(\cdot)}|h_\varepsilon|^{q(\cdot)}\right\|_{L^{p(\cdot)/q(\cdot)}},
\end{align*}
by \(1\le q(x)<1+\varepsilon\) on the support of \(h_\varepsilon\),  
\begin{align*}
2^{-\varepsilon}<\varrho_{\ell^{q(\cdot)}(L^{p(\cdot)})}\left(\frac{f_\varepsilon+g_\varepsilon}{2}\right)\le 1.
\end{align*}
Therefore,
\begin{align}\label{eq:ModularNOT-UC-3}
\varrho_{\ell^{q(\cdot)}(L^{p(\cdot)})}\left(\frac{f_\varepsilon+g_\varepsilon}{2}\right)\longrightarrow 1.
\end{align}
as $ \varepsilon\to 0 $. Keep in mind that \(q_-=1\) and $ \|({f_\varepsilon+g_\varepsilon})/{2}\|_{\ell^{q(\cdot)}(L^{p(\cdot)})}\le 1 $, Thus 
\begin{align*}
\varrho_{\ell^{q(\cdot)}(L^{p(\cdot)})}\left(\frac{f_\varepsilon+g_\varepsilon}{2}\right)\leq \|\frac{f_\varepsilon+g_\varepsilon}{2} \|_{\ell^{q(\cdot)}(L^{p(\cdot)})}^{q_-}\leq 1.
\end{align*}
It follows that 
\begin{align}\label{eq:NormNot-UC-3}
\| \frac{f_\varepsilon+g_\varepsilon}{2} \|_{\ell^{q(\cdot)}(L^{p(\cdot)})} \longrightarrow 1.
\end{align}
Therefore by \eqref{eq:NormNOT-UC-1}, \eqref{eq:NormNOT-UC-2}, \eqref{eq:NormNot-UC-3} and Definition \ref{defUC}, $ \ell^{q(\cdot)}(L^{p(\cdot)}) $ is not uniformly convex. At the same time, by \eqref{eq:ModularNOT-UC-1}, \eqref{eq:ModularNOT-UC-2}, \eqref{eq:ModularNOT-UC-3}, for $ r=1 $ in Definition \ref{defModUC}, the modular \(\varrho_{\ell^{q(\cdot)}(L^{p(\cdot)})} \) is not uniformly convex.
\end{proof}
\begin{corollary}
Consequently, in combination with Theorem \ref{theo:XisUC} under 
\begin{align*}
q(\cdot)\le p(\cdot)\le p_+<\infty,
\end{align*}
the condition \(q_->1\) is sharp for the uniform convexity of $ \ell^{q(\cdot)}(L^{p(\cdot)}) $. If $ q_->1,$ then $\ell^{q(\cdot)}(L^{p(\cdot)}) $ is uniformly convex, whereas at the endpoint $ q_-=1 $, the space \(\ell^{q(\cdot)}(L^{p(\cdot)})\) is no longer uniformly convex, and its modular \(\varrho_{\ell^{q(\cdot)}(L^{p(\cdot)})}\) is not uniformly convex either.
\end{corollary}

\section{Applications}\label{section5} 
The purpose of this section is to illustrate some consequences of the uniform convexity results obtained earlier. In particular, we derive applications to uniform non-squareness, renorming theory, and super-reflexivity of the space $\ell^{q(\cdot)}(L^{p(\cdot)})$ and related Besov spaces.

\subsection{Uniformly Non-Square Banach Spaces}

\begin{definition}\label{defUNonS}\cite{James1964}
A normed linear space $X$ is called \emph{uniformly non-square} if there exists a constant $\delta > 0$ such that for all $x,y \in S(X)$, either 
\begin{align*}
\left\| \frac{x+y}{2} \right\| \le 1 - \delta \quad \text{or} \quad \left\| \frac{x-y}{2} \right\| \le 1 - \delta,
\end{align*}
where $S(X) := \{x \in X : \|x\|_X = 1\}$ denotes the unit sphere of $X$.
\end{definition} 
This notion is closely related to uniform convexity, but involves only pairs of points on the unit sphere that are far apart. Uniform non-squareness is weaker than uniform convexity, although both are metric properties of Banach spaces. In particular, uniform non-squareness need not be preserved under equivalent renormings.

\begin{theorem}\cite{James1964}\label{theoUCisUnon-S}
Every uniformly convex Banach space is uniformly non-square.
\end{theorem}

\subsection{Renorming in Besov-Type Spaces}

In this subsection, we investigate reflexivity, super-reflexivity, and renorming properties of the variable mixed Lebesgue--sequence space $\ell^{q(\cdot)}(L^{p(\cdot)})$.

Let $E$ and $F$ be Banach spaces. We say that $F$ is \emph{finitely representable} in $E$ if for every $\varepsilon > 0$ and every finite-dimensional subspace $F^\circ \subset F$, there exists a finite-dimensional subspace $E^\circ \subset E$ with $\dim F^\circ = \dim E^\circ$ and an isomorphism $T : F^\circ \to E^\circ$ such that 
\begin{align*}
\|T\|\,\|T^{-1}\| < 1 + \varepsilon.
\end{align*}
Equivalently, every finite-dimensional subspace of $F$ can be embedded into $E$
with arbitrarily small distortion.

\begin{definition}\cite{B.Beauzamy}
A Banach space $E$ is called \emph{super-reflexive} if every Banach space finitely representable in $E$ is reflexive.
\end{definition}

\begin{definition}
The \emph{James constant} of a Banach space $X$ is defined by 
\begin{align*}
J(X) := \sup \left\{ \min\{\|x+y\|,\|x-y\|\} : x,y \in S(X) \right\}.
\end{align*}
\end{definition}

\begin{theorem}\label{theo1.3}\cite[Theorem~3.4]{Jamesconstant}
A normed linear space $X$ is uniformly non-square if and only if $J(X) < 2$.
\end{theorem}

\begin{proposition}
Let $ q_->1,~ p_+<\infty $ and $ q(\cdot)\leq p(\cdot) $ almost everywhere. Then the space $\ell^{q(\cdot)}(L^{p(\cdot)})$ is reflexive and satisfies 
\begin{align*}
J\bigl(\ell^{q(\cdot)}(L^{p(\cdot)})\bigr) < 2,
\end{align*}
where $J$ denotes the James constant.
\end{proposition}

\begin{proof}
By the uniform convexity of $\ell^{q(\cdot)}(L^{p(\cdot)})$ established earlier, Theorem~\ref{theoUCisUnon-S} implies that this space is uniformly non-square. Consequently, it is reflexive by \cite[Theorem~1.1]{James1964}. The inequality for the James constant then follows from Theorem~\ref{theo1.3}.
\end{proof}
Note that every super-reflexive Banach space is reflexive, whereas the converse is false in general.

\subsection{Super-Reflexivity of $\ell^{q(\cdot)}(L^{p(\cdot)})$}

It is well known that every super-reflexive Banach space admits an equivalent uniformly convex norm as well as an equivalent uniformly smooth norm. Moreover, a classical result due to Asplund shows that there exists an equivalent norm which is simultaneously uniformly convex and uniformly smooth. As a consequence of the uniform convexity established earlier, we obtain the following result.

\begin{corollary}\label{coro:renorming}
Let $p(\cdot), q(\cdot) \in \mathcal{P}$ such that $ q_->1,~ p_+<\infty $ and $ q(\cdot)\leq p(\cdot) $ almost everywhere. Then the space $\ell^{q(\cdot)}(L^{p(\cdot)})$, its dual, and its K\"othe dual are super-reflexive Banach spaces. In particular, each of these spaces admits an equivalent norm that is both uniformly convex and uniformly smooth.
\end{corollary}

\begin{proof}
By Theorem~\ref{theoUCisUnon-S}, $\ell^{q(\cdot)}(L^{p(\cdot)})$ is uniformly non-square, and hence super-reflexive by \cite[Chap.~III, \S1, Cor.~4]{B.Beauzamy}. Super-reflexivity of the dual space follows from \cite[Chap.~I, \S3, Cor.~8]{B.Beauzamy}. Since the dual space is isometrically isomorphic to the K\"{o}the dual \cite{GRS}, and super-reflexivity is invariant under isomorphisms \cite[Chap.~I, \S3, Cor.~6]{B.Beauzamy}, the K\"othe dual is also super-reflexive. The renorming assertions follow from standard results on super-reflexive spaces \cite{B.Beauzamy}.
\end{proof}

\subsection{Corresponding Besov Spaces}

The variable exponent Besov space $B^{s(\cdot)}_{p(\cdot)q(\cdot)}$ can be characterized via the mixed sequence space $\ell^{q(\cdot)}(L^{p(\cdot)})$.

Let $s(\cdot) \in L^\infty$ and Assume in addition that there exist
$c$, $p_\infty$, $q_\infty>0$
such that
\[
|p(x)-p_\infty| +|q(x)-q_\infty|\le\frac{c}{\log(e+|x|)}, \quad
\]
\[
|p(x)-p(y)| +|q(x)-q(y)| +|s(x)-s(y)|\le\frac{c}{\log\left(e+\frac{1}{|x-y|}\right)}
\]
for all $x,y \in {\mathbb R}^n$. Let $(\phi,\Phi) \in \mathcal{S}^2$ satisfy
\begin{equation}\label{eq:220201-1}
\operatorname{supp}(\widehat{\phi}) \subset \left\{\tfrac{1}{2} \le |\xi| \le 2 \right\}, \quad \operatorname{supp}(\widehat{\Phi}) \subset \{|\xi| \le 2\}, \quad |\widehat{\Phi}|^2 + \sum_{\nu=1}^\infty |\widehat{\phi}_\nu|^2 = 1.
\end{equation} 
where $\widehat{\phi},\widehat{\Phi}$ stand for Fourier transforms of real-valued even functions $\phi, \Phi$ and $\phi_\nu(\cdot)=2^{\nu n}\phi(2^\nu\cdot)$ and $\mathcal{S}$ denotes the Schwartz function space. Let $ \phi_0=\Phi $. The space $B^{s(\cdot)}_{p(\cdot)q(\cdot)}$ consists of all distributions
$ f \in \mathcal{S}' $ such that  
\begin{align*}
\|f\|_{B^{s(\cdot)}_{p(\cdot)q(\cdot)}}^\phi := \left\| (2^{\nu s(\cdot)} \phi_\nu * f)_{\nu=0}^\infty \right\|_{\ell^{q(\cdot)}(L^{p(\cdot)})} < \infty.
\end{align*}
The quasi-norm is independent (up to equivalence) of the choice of $(\phi,\Phi)$ satisfying \eqref{eq:220201-1} \cite[Theorem~5.5]{AlHa2010}. Moreover, as shown in \cite{GRS}, the operators 
\begin{align*}
A : f \mapsto (2^{\nu s(\cdot)} \phi_\nu * f)_{\nu=0}^\infty,
\qquad
B : (f_\nu)_\nu \mapsto \sum_{\nu=0}^\infty \phi_\nu * (2^{-\nu s(\cdot)} f_\nu)
\end{align*}
are bounded and satisfy $B \circ A = \mathrm{id}$. Moreover, $P=AB$ is a bounded projection on $\ell^{q(\cdot)}(L^{p(\cdot)})$, since $P^2=A(BA)B=P$. Its range is $A(B^{s(\cdot)}_{p(\cdot)q(\cdot)})=\ker(I-P)$, which is closed. Therefore, 
\begin{align*}
B^{s(\cdot)}_{p(\cdot)q(\cdot)}\cong A\big(B^{s(\cdot)}_{p(\cdot)q(\cdot)}\big).
\end{align*}
Since $\ell^{q(\cdot)}(L^{p(\cdot)})$ is super-reflexive, and super-reflexivity is inherited by closed subspaces and isomorphism, the Besov space $B^{s(\cdot)}_{p(\cdot)q(\cdot)}$ is also super-reflexive. In particular, it admits an equivalent norm that is both uniformly convex and uniformly smooth.

\subsection*{Ethics approval and consent to participate}
$ \newline $
Not applicable.
\subsection*{Consent for publication}
$ \newline $
Not applicable.
\subsection*{Competing Interests}
$ \newline $
The authors have no competing interests to declare that are relevant to the content of
this article.
\subsection*{Author contributions}
$ \newline $
Not applicable.
\subsection*{Funding}
$ \newline $
Not applicable.
\subsection*{Availability of data and materials}
$ \newline $
Data sharing not applicable to this article as no datasets were generated or analysed during the current study.

\bibliographystyle{amsplain}

\bibliography{GR2}

\end{document}